\documentclass{ModClass}
\usepackage[a4paper]{geometry}
\usepackage{relsize}

\usepackage{fdsymbol}

\newtheorem*{theorem*}{Theorem}

\newtheorem{theorem}{Theorem}[section]

\theoremstyle{definition}
\newtheorem{definition}[theorem]{Definition}
\newtheorem{example}[theorem]{Example}
\newtheorem{proposition}[theorem]{Proposition}
\newtheorem{corollary}[theorem]{Corollary}

\theoremstyle{remark}
\newtheorem{remark}[theorem]{Remark}

\numberwithin{equation}{section}

\newcommand{\set}[1]{\left\lbrace#1\right\rbrace} 

\begin{document}

\title{The Odd Hummer Principle}



\author{Sergio Alejandro Fernandez de soto Guerreo}
\email{sergio.fernandez@tugraz.at}
\thanks{}


\begin{abstract}
The Hummer Principle was born from the mind of Bob Hummer in 1946, which consists of performing card shuffles with an even number of cards while leaving some properties of the deck intact. In this document, we will present a generalization of this principle for an odd number of cards, along with ideas for applying and adapting magic tricks that use the Hummer Principle to a more general setting. Finally, we will discuss future work in this area.
\end{abstract}

\maketitle


\bibliographystyle{amsplain}

\section{Introduction}

In the realm of mathematical magic, there are well-known figures. One of whom is Bob Hummer, a magician from the last century whose presence left an indelible mark on the world of magic, and not only that, but also on mathematical magic. Hummer developed multiple principles and tricks throughout his life, and one of the most well-known of these is the so-called Hummer principle, which states the following in mathematical notation.

\begin{theorem*}[Hummer's principle]
Let be $n\in\mathbb{N}$. Consider a deck of $2n$ cards, all face down. After any number of CATO shuffles we have the following regularity:

\begin{itemize}
    \item The quantity of face-up cards in even-numbered positions is equal to the quantity of face-up cards in odd-numbered positions.
\end{itemize}

\end{theorem*}

This principle is known to be a tool that allows for shuffling and still maintaining control over much of the information of the cards, such as their orientation, color, value, among many others, using CATO shuffles, the acronym of \textbf{Cut And Turn Over} (see section 2).

However, since Bob Hummer first introduced this principle in 1946 (in ``Face up face down mysteries'' \cite{hummer}), no one had considered applying it to an odd number of cards in the literature, as the principle is formulated for $2n$ cards. Even renowned names in this field such as Colm Mulcahy, Woody Aragon, Pedro Alegría, Sergio Belmonte, Matt Baker, Leonard Rangel, among others, had not heard of applying the principle to an odd number of cards.

This manuscript presents such a generalization and some applications of it to magic. This was achieved by studying the Hummer principle from the perspective of group actions on sets to obtain shuffles analogous to CATO for different configurations of a deck of cards. Thanks to this approach to the principle, we arrived at what we will call Hummer invariants, DATO shuffles (acronym of Deal And Turn Over), and the Odd-Hummer principle.

\begin{theorem*}[Odd Hummer Principle]
Consider a deck with an odd number of face-down cards. After any number of DATO shuffles, we have the following:

\begin{itemize}
\item The face-up cards in even positions were originally cards in an odd position and vice versa, the face-up cards in odd positions were originally cards in an even position.
\end{itemize}
\end{theorem*}

Note that the Odd Hummer principle applies to an odd number of cards and yields an analog to the Hummer principle. This allows us to re-imagine and reconstruct the nearly 90 years of tricks and theory on CATO shuffles and the Hummer principle.

So let us now discuss a little more about the history of the Hummer principle, how the idea of generalization arose, what it consists of, and how the Odd Hummer principle can be applied.

\section{About the Hummer principle}

In this section, we will discuss the history behind the Hummer principle. To begin, let us define what Hummer shuffles, better known as \textit{CATO} shuffles (an acronym for Cut And Turn Over) are, which are defined on a deck of $2n$ cards, where $n \in \mathbb{N}$, as the following shuffles:

\begin{definition}[CATO shuffles]\label{def:cato-shuffle}
CATO shuffles consist of applying the following two shuffles in any order and as many times as desired:

\begin{itemize}
\item Cut the deck of cards at any point.
\item Turn over an even number of cards from the top as if it were a single card.
\end{itemize}
\end{definition}

In more formal terms, the first action is a left shift where the first element is moved to the last position; and the second movement can be thought of as reversing the order of the cards to be turned over and changing their orientation. For example, if we take the first four cards $(1,2,3,4)$ from a deck (array) of 6 cards $(1,2,3,4,5,6)$ and turn them over as a single card, we obtain $(1,2,3,4,5,6) \mapsto (\overline{4},\overline{3},\overline{2},\overline{1},5,6)$ where the bar denotes the change in orientation of the card. With this in mind, we can state the Hummer principle as follows.

\begin{theorem}[Hummer principle \cite{persigraham}]\label{theorem:hummer-principle}
For $n\in\mathbb{N}$. Consider a deck of $2n$ cards, all face down. After any number of CATO shuffles, the following regularity holds:

\begin{itemize}
\item The number of face-up cards in even-numbered positions is equal to the number of face-up cards in odd-numbered positions.
\end{itemize}

\end{theorem}

Given the Hummer principle, let us discuss where it originated from. To do so, we establish the following preliminary result.

\begin{theorem}[Red-black separation principle]\label{theorem:red-black-separation}
Given a deck with n red cards and n black cards, which is shuffled and separated into two piles with the same number of cards, the following holds:
\begin{itemize}
\item The number of red cards in one pile is equal to the number of black cards in the other pile.
\end{itemize}
\end{theorem}

\begin{proof}
Consider that after shuffling the $2n$ cards, $n$ cards of red color and $n$ of black color, the deck of size $2n$ is separated into a deck $A$ and a deck $B$ of size $n$; if there are $x$ red cards in deck $A$, then by hypothesis there are a total of $n$ cards of that color, which means that the remaining $n-x$ red cards will be in deck $B$, and therefore there will be $x$ black cards in deck $B$. That is, in deck $A$ there are $x$ red cards and $n-x$ black cards, and in deck $B$ there are $x$ black cards and $n-x$ red cards.

\end{proof}

The Red-black separation principle was first applied by the magician Stewart James in his tricks ``Tapping a Brain Wave" and ``The Psychic Pickpocket" in 1938 \cite{james}, and in the following years Bob Hummer showed great creativity in using this principle, which led to the development of the Hummer principle over time.

\begin{remark}
   The Hummer principle can be derived from Theorem \ref{theorem:red-black-separation}, and to prove the former, we only need to consider the definition of the \textit{CATO} shuffle and the parity of the position. 
    
\end{remark}

There is something important to consider about the red/black separation principle and the Hummer principle: they are both applied to an even number of elements. It is natural to wonder about these two results for an odd number of elements.

In terms of magic, there is no trick where something is done with an odd number of elements in a way that uses something analogous to Theorem \ref{theorem:hummer-principle}; in terms of mathematics, the situation is not encouraging either, since the few that consider this case always refers to the fact that the principle does not work for an odd number of cards. Therefore, looking at the principle from a not-yet-fully-explored perspective helps to extend these notions of the principles.

\section{\textit{CATO} shuffles as group actions}\label{section:cato-actions}

In the previous section, we defined the \textit{CATO} shuffles (Definition \ref{def:cato-shuffle}), which can be regarded as two actions of a group on a set, following Ensley's line of thought \cite{invariant}. The first action is to cut the deck of cards wherever desired; the second is to flip an even number of cards as a single unit. Let us first study these two actions of the \textit{CATO} shuffles.

Let $D_{2n}:=\set{1, 2, \ldots, 2n, \overline{1}, \ldots, \overline{2n}}$ be the set of $2n$ cards with their two possible orientations, where the number $i$ indicates the initial position of each card and the bar indicates a change in orientation; from now on we will use $D=D_{2n}$ as notation unless otherwise stated. Let $\alpha: D \to D$ be defined by $(1, 2, \ldots, 2n) \mapsto (2, 3, \ldots, 2n, 1)$, a left shift in a cycle. Cutting the deck of cards at any point, say at the $j$-th card, would result in the following:
\begin{equation*}
(1, 2, \ldots, j, \ldots, 2n) \mapsto (j+1, j+2, \ldots, 2n, 1, 2, \ldots, j),
\end{equation*}
and we can easily see that the first condition of the \textit{CATO} shuffles is equivalent to applying $j$ times $\alpha$.

If we consider the second action of the \textit{CATO} shuffles, which is to flip an even number of cards, let it be the first $k$ cards with $k$ even, then we can view it as the following map $\delta_k:D\to D$ where
\begin{equation*}
(1, \ldots, k, k+1, \ldots, 2n) \mapsto (\overline{k}, \overline{k-1}, \ldots \overline{2}, \overline{1}, k+1, k+2, \ldots, 2n),
\end{equation*}

The map $\delta_k$ is not actually described in this way in Hummer's notes, but is initially presented as flipping only the top two cards and cutting the deck. Therefore, we can define the following map $\beta: D \to D$ such that $(1, 2, \ldots,2n)\mapsto(\overline{2}, \overline{1}, 3, \ldots, 2n)$.

From this previous notation, we can see the following result.

\begin{proposition}\label{proposition:cato-as-actions}
The action $\delta_k$ of flipping an even number $k$ of cards from the top as a single card in the \textit{CATO} shuffles, can be obtained by compositions of $\alpha$ and $\beta$.
\end{proposition}

\begin{proof}
We will read the compositions of the permutations from left to rigth. Let $k\leq2n$ be the number of cards to flip. If $k=2$, $\delta_2=\beta$, which is sufficient. If $k=4$, $\delta_4 = \beta(\alpha\beta)^2\alpha^{-2}\beta\alpha\beta\alpha^{-1}\beta$; if $k=6$, $\delta_6 = \beta(\alpha\beta)^4\alpha^{-4}\beta(\alpha\beta)^3\alpha^{-3}\beta(\alpha\beta)^2\alpha^{-2}\beta\alpha\beta\alpha^{-1}\beta$; where we can already see a regular pattern, so proceeding by induction we have that for $k\geq 2$, $\delta_k = \prod_{i=0}^{k-2}\beta_{i}(\alpha\beta)^{i}\alpha^{-i}$, where $\beta_i=\beta$ if $i\neq0$ and $\beta_0=I_d=\beta^0$, which completes the proof.
\end{proof}

\begin{remark}
In the above proof, it should be noted that: first, the maps $\alpha$ and $\beta$ do not commute; second, the map $\delta_k$ can be constructed in other ways, the ones given in the previous proof were the most regular for the inductive writing of $\delta_k$. Furthermore, we can see what happens with $k\geq0$:
\begin{itemize}
    \item $\prod_{i=0}^{k}\beta_{i}(\alpha\beta)^{i}\alpha^{-i}((1, \ldots, 2n))=(\overline{k}, \ldots,\overline{2},\overline{1}, k+1, \ldots, 2n)$, if k is even.
    \item $\prod_{i=0}^{k}\beta_{i}(\alpha\beta)^{i}\alpha^{-i}((1, \ldots, 2n))=(k, \ldots, 2, 1, k+1, \ldots, 2n)$, if $k$ is odd.
\end{itemize}
\end{remark}

So far in this section, we have been able to reduce the CATO shuffles to the maps $\alpha$ and $\beta$, since we can see the CATO shuffles as any action in the group $\langle\alpha,\beta\rangle$, which is actually a subgroup of what is known as the signed permutations, which we can be define as a permutation $\pi$ of the set $\lbrace n,-n+1,\ldots,-1,1,2,\ldots,n\rbrace$, one example for $n=6$ is $\pi=2,-3,6,-1,4,5$, in this case we can think the sign as a orientations of the cards.

The goal now is to consider this subgroup as acting on the set $D$, and we will be interested in the invariants of this action.

\section{Invariants under \textit{CATO} shuffles}

In \cite{invariant}, we can see a study about \textit{CATO} shuffles acting on a set and the invariants on the deck of cards, but only for the case with four cards, which in magic is known as the \textit{Baby Hummer}, a magic trick created in 1968 by the magician Charles Hudson. This study is carried out by identifying three permutations, $\alpha$ and $\beta$, as we have described in the previous chapter, along with a permutation $\delta$ that takes $(1,2,3,4)\to(\overline{4},\overline{3},\overline{2},\overline{1})$, but we will ignore $\delta$ since $\delta=\delta_4\in\langle\alpha,\beta\rangle$.

In \cite{invariant}, there are results that we will extend to an odd number of cards, and for that, it is important to consider the following.

\begin{definition}[Maloriented cards]
For $D'_{n}=\set{1,2,\ldots,n,\overline{1},\ldots, \overline{n}}$ and an array $A$ of $n$ elements in $D'^n$, we will say that there are $i$ maloriented cards if there are $n-i$ cards with the same orientation, with $i<n-i$.
\end{definition}

Let us illustrate this concept with an example. Consider the following array: $(\overline{5},\overline{4},\overline{1},3,6,\overline{2})$. The maloriented cards are 3 and 6. In the array $(\overline{1},3,4,5,6,\overline{2})$, the maloriented cards are 1 and 2.

Now we talk about the Baby Hummer. In mathematical terms, the trick consists of applying the CATO shuffles to a specific array of 4 cards, $(\overline{1},2,3,4)$, and then applying a final action that leave only the card 3 as a maloriented card.

The study of the Baby Hummer \cite{invariant} consists of taking sets of arrays that are invariant under the action of $H=\langle\alpha,\beta\rangle$ for $D_4=\set{1,2,3,4,\overline{1},\overline{2},\overline{3},\overline{4}}$, and using the previous definition we obtain the following results:

\begin{proposition}
Let $C$ be the set of arrays with only one maloriented card. Then, under the action of $H$, the set $C$ always have a one maloriented card.
\end{proposition}

\begin{proposition}
Without loss of generality, consider the four-card array $(\overline{1},2,3,4)$. Under the action of $H$, the card 3 is always two positions away from the maloriented card.
\end{proposition}

These propositions indicate that there are certain conditions, we call it Hummer invariants, that no matter how many times the \textit{CATO} shuffles are performed, will remain invariant. Specifically, the regularity of the Hummer principle (Theorem \ref{theorem:hummer-principle}) is maintained, which refers to the relationship between the number of face-up and face-down cards, and their relationship to the parity of their position. This will be the key to generalize the result to an odd number of cards.

\section{Generalization}

The ideas presented in the previous sections  will lead us to what we will define as \textit{Deal And Turn Over} (\textit{DATO}) shuffles for a deck of any number of cards.

\begin{definition}[\textit{DATO} shuffles]\label{def:dato}
\textit{DATO} shuffles consist of applying the following two shuffles in any order and as many times as desired:

\begin{itemize}
\item Dealing an even number of cards and completing the cut.
\item Turning over an even number of cards from the top as if they were a single card.
\end{itemize}
\end{definition}

Unlike \textit{CATO} shuffles, \textit{DATO} shuffles do not cut the deck at a given point, but from a card trick perspective, dealing in a certain way, does cut the deck. Therefore, we need to establish the property that will remain invariant under \textit{DATO} shuffles, as follows:

\begin{itemize}
\item The face-up cards in even positions were originally cards in odd positions, and vice versa. The cards face-up in odd positions were originally cards in even positions.
\end{itemize}

Note that the previous condition is an invariant under the action of \textit{CATO} shuffles if applied to an even number of cards. For an odd number of cards, \textit{CATO} shuffles no longer preserve this condition because if $\alpha$ is applied, the parity of all cards changes, except for the parity of the first card. \textit{DATO} shuffles are a solution to ``track'' the initial parity of the position of the cards, depending on whether they are face-up or face-down, when dealing with an odd number of cards.

Consider performing the shuffles in Definition \ref{def:dato} as actions on the set $D_n'=\set{1,\ldots , n, \overline{1},\ldots, \overline{n}}$, where $n$ is an odd number. From now on, we will use $D'=D_n'$ when $n$ is odd. Flipping $k$ cards, with $k$ even, as a single one, the action remains the $\delta_k$ shuffle, but applied to the set $D_n'$, and dealing an even number of cards is the following action: for even $j$, we define $\gamma_j: D' \to D'$ such that
$$(1,2,\ldots,j,\ldots,n)\mapsto(j+1,\ldots,n,j,j-1,\ldots,2,1).$$

The \textit{DATO} shuffles can be regarded now as the group generated by $\delta_k$ and $\gamma_j$, in other words as the group $\langle \delta_k,\gamma_j \rangle$, where $k,j\leq n$.

\begin{theorem}[Odd Hummer Principle]\label{theorem:dsoto-principle}
Let $n\in\mathbb{N}$ be odd. For a deck of $n$ cards, all face down. After any number of DATO shuffles, the following holds:

\begin{itemize}
\item The cards facing up in even positions were originally in odd positions, and vice-versa: The cards facing up in odd positions were originally in even positions.
\end{itemize}

\end{theorem}

\begin{proof}

Note that if a card is in position $i$, after applying $\delta_k$ with $k$ even, we have that

$$\delta_k(i)=
    \begin{cases}
        \overline{k+1-i} & i \leq k\\
        i & \text{if } i > k
    \end{cases}$$

where the position of $\delta_k(i)$ is even (odd) if $i$ is odd (even).

Now, when applying $\gamma_j$ with even $j$, we have that

$$\gamma_j(i)=
    \begin{cases}
        i-j & \text{if } i \geq j\\
        n+1-i & i < j
    \end{cases}$$

where the position of $\gamma_j(i)$ is even (odd) if $i$ is even (odd).
 
\end{proof}

Theorem \ref{theorem:dsoto-principle} implies that all the results and tricks with the Hummer principle can be done with the Odd Hummer principle, and consequently it opens the door to think about how to extrapolate everything developed with the Hummer principle to an odd number of cards.

\begin{remark}
As previously discussed, in the \textit{DATO} shuffles compared to the \textit{CATO} shuffles, the ability to cut the deck and perform the permutation $\gamma_j$ is lost. However, other actions can be taken which are similar to performing a cut:

\begin{itemize}
    \item Take an odd number of cards from the top, flip them over, and place them on the bottom. Let $l$ be an odd number, and define $\lambda_l:D' \to D'$ such that
    $$(1,2,\ldots,l,\ldots,n)\mapsto(l+1,\ldots,n,\overline{l},\overline{l-1},\ldots,\overline{2},\overline{1}).$$
    
    \item Also we can use the following action. Let $l$ be an odd number, $\omega_l:D' \to D'$ such that:
    $$(1,2,\ldots,l,\ldots,n)\mapsto(l+2,\ldots,n,l+1,1,2,\ldots,l).$$
    
\end{itemize}

There are more actions that preserve the idea of cutting the deck.

\end{remark}

\section{Some magical ideas}
In the previous section, much emphasis was placed on finding shuffles analogous to \textit{CATO} from the point of view of being able to have cuts in the deck. In consideration for use in magic, having this ``freedom'' to cut the deck contributes greatly to the construction of card tricks. Additionally, having a way to cut most tricks that use \textit{CATO}, they can be adapted to \textit{DATO} much more easily.

In this vein, this chapter will present two other types of shuffles that can take advantage of all tricks related to the Hummer principle.

First, we will consider that each card in an array has three attributes. In this case, we will consider its orientation (face-up or face-down), its color (red or black), and the parity of its position in an array (even or odd). Taking this into account, if we start with an initial array of $n$ face-down cards, alternating by color, the Hummer and Odd Hummer principles applied to this array give us the following result:

\begin{corollary}\label{coro:tree-caulities}
Given a deck of $n$ face-down cards, with the cards intercalated by color, without loss of generality, red, black, red, black, etc. If any amount of DATO shuffles is applied when $n$ is odd, or CATO shuffles when $n$ is even, the following regularity holds:
\begin{itemize}
\item Knowing the properties of a card before the shuffle, then after the shuffles if we know two of three properties of the card (color, orientation, parity of position) we can know the last one.
\end{itemize}

\end{corollary}

\begin{example}
As we know, a standard modern deck of cards has four suits, two of them are black, spades ($\spadesuit$) and clubs ($\clubsuit$), and two of them are red, diamonds ($\diamondsuit$) and hearts ($\heartsuit$).

Let us suppose our initial array is $(A\spadesuit,2\heartsuit,3\spadesuit,4\heartsuit,5\spadesuit,6\heartsuit)$, with the cards facing down. If we flip two cards together ($\delta_2$), then cut one card ($\alpha$), and finally flip four cards ($\delta_4$), we obtain
$$(\overline{5\spadesuit},\overline{4\heartsuit},\overline{3\spadesuit},A\spadesuit,6\heartsuit,\overline{2\heartsuit})$$

Now, if we know that the cards that were facing down initially in even positions were red cards, it means that the fifth card, which is facing down, will be a red card at the end. The same applies for the other two properties, orientation and parity of the position.
\end{example}

\subsection*{DATO and CATO together}

Another observation to consider (personally my favorite) is the fact that if you have a deck with $n$ cards, where $n$ is even (odd), and it is divided into two piles of odd (even) quantities, you can apply the \textit{DATO} (\textit{CATO}) shuffles to the new decks.

\begin{example}
Suppose we have an array of twenty cards to which we apply the \textit{CATO} shuffles. Then, at some point, if we cut a stack of 7 cards and leave it on the table, we can now apply the \textit{DATO} shuffles to the remaining 13 cards in our hand. Finally, we can recombine the pile on the table with the cards in our hand to have the full deck of twenty cards, and the Corollary \ref{coro:tree-caulities} will still hold.
\end{example}

\subsection*{More than DATO}

Finally, let's talk about one last variation of the invariants we have seen under a modification to the shuffling with regard to the cards that are turned over. What would happen if instead of an even number of cards being flipped, an odd number is flipped?

If we think now of $m$ as an odd number and the map $\tau_m$ as the action on an arrangement that does the following:

\begin{equation*}
(1, \ldots, m, m+1, \ldots, n) \mapsto (\overline{m}, \overline{m-1}, \ldots \overline{2}, \overline{1}, m+1, m+2, \ldots, n),
\end{equation*}

We can describe new groups that generate new actions, $\langle \tau_m, \alpha\rangle$ and $\langle \tau_m, \gamma_j\rangle$, depending on whether we want to perform actions on an arrangement with an even or odd number of cards.

In this case, we can rephrase Theorems \ref{theorem:hummer-principle} and \ref{theorem:dsoto-principle} as follows:

\begin{theorem}
Let be $n\in\mathbb{N}$. Consider a array of $n$ cards, all face down and interspersed by one binary property (for instance: colors). Under the action of $\langle \tau_m, \alpha\rangle$ if $n$ is even or $\langle \tau_m, \gamma_j\rangle$ if $n$ is odd, the following regularity holds:

\begin{itemize}
    \item All the cards in even positions will have the same property value, and odd positions cards will have the other value of the property.
\end{itemize}
\end{theorem}

\begin{example}
    Suppose the cards in an array of 17 cards are arranged as red, black, red, black, and so on. In this case, after applying any element of $\langle \tau_m, \gamma_j\rangle$ (shuffle), if the first card is a black card, then all the cards in odd positions will be black, and the cards in even positions will be red. The same would happen if we had 20 cards, for example, and used $\langle \tau_m, \alpha\rangle$.
\end{example}

 As we can see, thinking about invariants using the \textit{DATO} and \textit{CATO} shuffles opens up a world of possibilities, but there is still much to explore.

\section{Open questions}

The literature on magic tricks that use the Hummer principle is vast, and viewing the principle as the invariants it leaves in a deck of cards opens the door to several open questions:

\begin{itemize}
    \item Are there other types of shuffles that preserve the Odd Hummer principle that are not in the same group? In other words, is there another group different from $\langle \gamma_j, \delta_k\rangle$ that can be used in the principle?
    
    \item The \textit{CATO} shuffles can be seen as the group $\langle \alpha,\beta\rangle$, which has only two generators. So, it is natural to ask, what is the minimum number of generators for the \textit{DATO} shuffles $\langle \gamma_j, \delta_k\rangle$ ? What are those generators?
    \end{itemize}

In section \ref{section:cato-actions} we saw that \textit{CATO} and \textit{DATO} shuffles in the Hummer and Odd Hummer principles are subgroups of the group of signed permutations.
\begin{itemize}
    \item What other kind of permutations can be used to generalize the idea of the principles to objects that do not have only two orientations, or where the positions to consider are module $j$, or where the properties have more than two values?
\end{itemize}

There is still a long way to go, but this paper is a first step into the world of the multiple ways in which the Hummer principle can be generalized.

\bibliography{references}

\providecommand{\bysame}{\leavevmode\hbox to3em{\hrulefill}\thinspace}
\providecommand{\MR}{\relax\ifhmode\unskip\space\fi MR }
\providecommand{\MRhref}[2]{%
  \href{http://www.ams.org/mathscinet-getitem?mr=#1}{#2}
}
\providecommand{\href}[2]{#2}
\begin{thebibliography}{1}

\bibitem{invariant}
Douglas~E. Ensley, \emph{Invariants under group actions to amaze your friends},
  Mathematics Magazine (1999), 383--387.

\bibitem{persigraham}
P.~Diaconis;~R. Graham, \emph{Magical mathematics: The mathematical ideas that
  animate great magic tricks}, Princeton University Press, 2011.

\bibitem{hummer}
Bob Hummer, \emph{Another dozen hummer}, Magic World Publishers, 2019.

\bibitem{james}
A.~Slaight, \emph{The james file}, Hermetic Press, 2000.

\end{thebibliography}

\end{document}